\documentclass[12pt,leqno,a4paper]{amsart}
\usepackage{amssymb,enumerate}

\usepackage{enumerate}
\usepackage{xcolor}

\newtheorem{theorem}{Theorem}[section]
\newtheorem{lemma}[theorem]{Lemma}

\newtheorem*{thmA}{Theorem A}
\newtheorem*{thmC}{Theorem C}
\newtheorem*{corB}{Corollary B}

\numberwithin{equation}{section}

\begin{document}
\title[A criterion for a normal subgroup to be hypercentral]{A criterion for a normal subgroup to be hypercentral based on class sizes}

    \author[Beltr\'an]{
       Antonio Beltr\'an\\
     Departamento de Matem\'aticas\\
      Universitat Jaume I \\
     12071 Castell\'on, Spain\\
        }

 \thanks{ Antonio Beltr\'an: abeltran@uji.es,                                                                                   ORCID: 0000-0001-6570-201X }



\begin{abstract}
Let $G$ be a finite group and $N$ a normal subgroup of $G$. We prove that the knowledge of the sizes of the conjugacy classes of $G$ that are contained in $N$ and
 of their multiplicities  provides information of $N$ in relation to the structure of $G$. Among other results, we obtain a criterion to determine whether a Sylow $p$-subgroup of $N$   lies in the hypercentre of $G$ for a fixed prime $p$, and therefore, whether the whole subgroup $N$ is hypercentral in $G$.

\end{abstract}

\maketitle

\section{Introduction}

It is well known that the set of conjugacy class sizes of a finite group may provide structural  information of it.  There exist multiple relevant results  that endorse it, for instance, It\^o's theorem about the nilpotency of groups having exactly two class sizes. The reader is referred to  \cite{CC}  for a good survey on this widely studied research topic.  But the knowledge of only the class sizes of a group may result quite restrictive  to get to know certain properties of it. In fact, the order of the group is unbounded, and its nilpotency or solvability cannot be determined either (\cite{CC2} and \cite{Nav}, respectively).   Nevertheless, when the multiplicities of the class sizes are provided as well, it seems that further information of the group could be obtained.   Indeed, a reference  work  by Cossey, Hawkes and Mann \cite{CHM} establishes that nilpotency  can be read  off from the knowledge of the class sizes and their multiplicities, however, recognising solvability and supersolvability of a group  still continue to be open problems.

Let $G$ be a finite group and $N\unlhd G$.  Instead of looking into all conjugacy classes of $G$, we restrict ourselves to  the conjugacy classes of $G$ contained in $N$, which are called $G$-classes of $N$. There also exist research works (see  \cite{AkBF} or \cite{ABF} for instance) showing that  a limited set of $G$-class sizes of $N$  may exert a strong influence on the structure of $N$. Our aim in this paper is to investigate  which properties of $N$  can be  extracted from  the $G$-class sizes together with their multiplicities. For this purpose,  inspired by the class size frequency function introduced in \cite{CHM}, we define the {\it $G$-class size frequency function} of $N$, $w^G_N: {\mathbb N}\longrightarrow {\mathbb N}$, as follows:
   $$w^G_N(n)=\frac{1}{n}|\{g\in N : |g^G|=n\}|,$$
that is, for every $n\in {\mathbb N}$, $w^G_N(n)$ is the number of  $G$-classes contained in $N$ that have cardinality $n$. This function trivially provides arithmetical
properties of $N$; its order and hence, the order of its Sylow subgroups, as well as the number of elements of $N$ that are central in $G$. No information of  $G$, whenever $N<G$, seems to be able to be recovered.  Even though  the $G$-class sizes and their multiplicities may involve primes that do not divide the order of $N$, we show, however, that the $G$-class size frequency function of $N$  can still provide   properties of $N$ with regard to its immersion in $G$. Our first result is

\begin{thmA} Let $G$ be a finite group and $N$ a normal subgroup of $G$. Then the $G$-class sizes of $N$ and their multiplicities determine whether $N$ is hypercentral in $G$.
\end{thmA}

 By using properties of the hypercentre of a group, Theorem A yields to the following corollary, which shows how the $G$-class size frequency function of a normal subgroup  $N$ gives information about  the external behaviour  of the prime-power order elements of $N$ with respect to $G$.

\begin{corB} Let $G$ be a finite group, $N$ a normal subgroup and $p$ a prime. Then the $G$-class sizes of $N$ and their multiplicities determine whether all $p$-elements of $N$ belong to ${\bf C}_N({\bf O}^p(G))$.
\end{corB}

Inspired by \cite{Mat}, we obtain the next result relative to commutativity in $G$ of certain prime-power order elements of $N$.

\begin{thmC} Suppose that $G$ is a finite group and $N$ is a normal subgroup of $G$. Then the $G$-class sizes of $N$ and their multiplicities determine whether all elements of $N$ lying in the centre of some Sylow subgroup of $G$ belong to ${\bf Z}(G)$.
  \end{thmC}

We will show by an example that the $G$-class size frequency function of a normal subgroup does not detect whether such subgroup is abelian, nor the derived length when the normal subgroup is solvable or the nilpotency class when it is nilpotent. We have been unable to decide, however, whether nilpotency can be recognised. Of course, determining solvability or supersolvability, as happens with  class sizes and their multiplicities in a group, are likewise open problems. We encourage the reader to continue with this research.

\section{Notation and Preliminaries}
Let $G$ be a finite group and $N$ a normal subgroup. For our convenience, if  $\pi$ is a set of primes, we will write $\mathcal{S}_{\pi}^G(N)$  to denote the subset of elements of $N$ whose $G$-class has cardinality a $\pi$-number. The cardinality of this set is easily computed from  the $G$-class size frequency function of $N$, $w^G_N$, defined in the Introduction, by the formula
$$|\mathcal{S}_{\pi}^G(N)|=\sum_{n_{\pi'}= 1} w^G_N
(n) \cdot n,$$
where $n_{\pi'}$ denotes the $\pi'$-part of $n\in{\mathbb N}$.  We will specially focus on the cases $\pi=\{p\}$ and $\pi=\{p\}'$  with $p$ prime, and will write $\mathcal{S}_p^G(N)$ and $\mathcal{S}_{p'}^G(N)$  to denote the corresponding sets.

\bigskip
 We  recall the definition of the hypercentre of a group \cite[p. 7-8]{Huppert}. Let $G$ be a finite group and $1\leq {\bf Z}_1(G)\leq {\bf Z}_2(G)\leq\ldots$ be a series of subgroups of $G$,  where ${\bf Z}_1(G)$ is the centre of $G$ and ${\bf Z}_{i+1}(G)$ for $i\geq 1$ is  defined by ${\bf Z}_{i+1}(G)/{\bf Z}_i(G)={\bf Z}(G/{\bf Z}_i(G))$. Then the hypercentre of $G$, ${\bf Z}_{\infty}(G)$, is the last term of this series, and is nilpotent and characteristic in $G$. A group $G$ is nilpotent if and only if $G={\bf Z}_{\infty}(G)$.

 Before giving the proofs of our results, we need to state and develop several basic lemmas about hypercentres and normal subgroups.
The first property is renowned and the second, which originally appeared in  a slightly modified form  in \cite[Lemma 2]{Peng}, is a consequence of it, which is often more convenient to apply.

\begin{lemma}[p. 739, \cite{Huppert2}] \label{l1} Let $G$ be a finite group. A normal $p$-subgroup $N$ of $G$ lies in ${\bf Z}_\infty(G)$ if and only if $G/{\bf C}_G(N)$ is a $p$-group.
\end{lemma}

\begin{lemma}[Lemma 2, \cite{CHM}] \label{l2} A $p$-element $x$ of a finite group $G$ lies in ${\bf Z}_\infty(G)$ if and only if $x$ commutes with each $p'$-element of $G$.
\end{lemma}

It is well known that the primes dividing $|{\bf Z}_{\infty}(G)|$ are exactly  those dividing $|{\bf Z}(G)|$. The same occurs when one intersects with any normal subgroup of $G$.

\begin{lemma} \label{l3} Let $G$ be a finite group and $N$ a normal subgroup of $G$. If a prime $p$ divides $|N\cap {\bf Z}_\infty(G)|$, then $p$ divides $|{\bf Z}(G)\cap N|$.
\end{lemma}

\begin{proof} Suppose that $p$ divides $|N\cap {\bf Z}_\infty(G)|$ and take $P_0$ a Sylow $p$-subgroup of $N\cap {\bf Z}_\infty(G)$. Since $N\cap {\bf Z}_\infty(G)$ is nilpotent and normal in $G$, we have $P_0\unlhd G$. Now,  if $P$ is a Sylow $p$-subgroup of $G$,  it is clear that $P_0\leq P$, and hence $P_0\cap {\bf Z}(P)\neq 1$. By applying Lemma (\ref{l1}) to $P_0$, we get $G=P{\bf C}_G(P_0)$ and this implies that $1\neq P_0\cap {\bf Z}(P)\leq {\bf Z}(G)$. Since we also have $P_0\cap{\bf Z}(P)\leq N$, we conclude that $p$ divides $|N\cap {\bf Z}(G)|$.
\end{proof}

\section{Proofs}

We are ready to prove our results. We remark that our proof of Theorem A, which we state again in a more convenient form,   differs from that   originally appeared in \cite{CHM} to demonstrate that the nilpotency of a group can be recognised from the class size frequency function. We have adapted instead the ideas of \cite[Theorem 23.5]{Huppert}.

\begin{theorem} \label{th1}
Let $G$ be a finite group and $N$ a normal subgroup of $G$. Then the $G$-class sizes of $N$ and their multiplicities determine whether $N$ is hypercentral in $G$. Precisely, if $p$ is a prime number, then
\begin{itemize}

\item[(a)] $|\mathcal{S}_{p}^G(N)|_p=|N\cap {\bf Z}_{\infty} (G)|_p$, where ${\bf Z}_{\infty} (G)$ is the hypercentre of $G$.
\item[(b)] a Sylow $p$-subgroup of $N$ is hypercentral in $G$ if and only if $|\mathcal{S}_{p}^G(N)|_p=|N|_p$.
\end{itemize}
\end{theorem}

\begin{proof}   We proceed by induction on $|N|$ to prove (a).  By definition  of the set $\mathcal{S}_{p}^G(N)$ we have
$$|\mathcal{S}_{p}^G(N)|\equiv |N\cap {\bf Z}(G)| \quad  ({\rm mod } \, p). $$
Suppose first that $p$ does not divide $|N\cap {\bf Z}_{\infty} (G)|$, so certainly $p$ does not divide $|N\cap {\bf Z}(G)|$ either. In this case the above congruence proves that
$|\mathcal{S}_{p}^G(N)|$ is a $p'$-number. Hence   $|\mathcal{S}_{p}^G(N)|_p=1= |N\cap {\bf Z}_{\infty} (G)|_p$, so (a) is proved.

Henceforth, we will assume that $p$ divides  $|N\cap {\bf Z}_{\infty} (G)|$. By Lemma (\ref{l3}), we have that $p$ divides $|N \cap {\bf Z}(G)|$, so we can take $M\leq N\cap {\bf Z}(G)$
such that $|M|=p$. We consider $\overline{G}=G/M$ with the normal subgroup $\overline{N}=N/M$. In the following we will prove

\begin{equation}\label{eq}
 |\mathcal{S}_{p}^{\overline{G}}(\overline{N})| \cdot p = |\mathcal{S}_{p}^G(N)|
\end{equation}

For every $g\in N$, we have
$${\bf C}_G(g)/M\leq {\bf C}_{\overline{G}}(\overline{g})=\overline{U}.$$
If $u\in U$, then we can write $g^u=gf(u)$, with $f(u)=[g,u]\in M$. As $M$ is central in $G$, then
$$f(uv)=[g, uv]=[g, v][g, u]^v=[g,u][g,v]=f(u)f(v),$$
for all $v\in U$. Thus, $f$ is a group homomorphism from $U$ to $M$, and its kernel is exactly ${\bf C}_G(g)$. Consequently,
$|U/{\bf C}_G(g)|$ divides $|M|=p$.

Take now $\overline{g}\in \mathcal{S}_{p}^{\overline{G}}(\overline{N})$, that is, take $\overline{g}\in \overline{N}$ satisfying $|\overline{g}^{\overline{G}}|=p^i$ for some $i\in {\mathbb N}$. By the above paragraph we know that  either $U={\bf C}_G(g)$ or $|U|=p|{\bf C}_G(g)|$. In the former case we have
$$|g^G|=|G:{\bf C}_G(g)|=|\overline{G}:{\bf C}_{\overline{G}}(\overline{g})|=p^i,$$
and in the latter,
$$|g^G|=|G:{\bf C}_G(g)|=p|G:U|= p|\overline{G}:{\bf C}_{\overline{G}}(\overline{g})|=p^{i+1}.$$

We consider the set of $\overline{G}$-classes of $\overline{N}$ of cardinality $p^i$, whose number is given by the $\overline{G}$-class frequency function, $w_{\overline{N}}^{\overline{G}}(p^i)$, and  the  correspondence, by taking bars,  between $G$-classes of $N$ and $\overline{G}$-classes of $\overline{N}$. Next we see that every $\overline{G}$-class of $\overline{N}$, with regard to pre-images,   gives rise to $G$-classes of $N$ in two different ways. 

Write $M=\langle z\rangle$ (with $z$ of order $p$) and let $\overline{g}\in \mathcal{S}_{p}^{\overline{G}}(\overline{N})$.  We have  seen above that either $|\overline{g}^{\overline{G}}|=|g^G|$ or $p|\overline{g}^{\overline{G}}|=|g^G|$. In the first case, ${\bf C}_{\overline{G}}(\overline{g})=\overline{{\bf C}_G(g)}$, and this easily leads to that $g^G$, $(gz)^G$, $\ldots$ , $(gz^{p-1})^G$ 
 are  $p$ distinct $G$-classes of $N$ that are pre-images  of ${\overline{g}}^{\overline{G}}$, with the same cardinality as 
 ${\overline{g}}^{\overline{G}}$. If $p|\overline{g}^{\overline{G}}|=|g^G|$, then the inverse image of ${\overline{g}}^{\overline{G}}$ is exactly the $G$-class of $N$, $g^G$, with cardinality  $p|{\overline{g}}^{\overline{G}}|$.

 Accordingly, for every $i\geq 0$, we  can decompose $w_{\overline{N}}^{\overline{G}}(p^i)$ into two summands
$$w_{\overline{N}}^{\overline{G}}(p^i)= w_1(p^{i}) +w_2(p^i),$$
where $w_1(p^{i})$ and $w_2(p^i)$ denote the number of $\overline{G}$-classes of $\overline{N}$ corresponding to each one of both types of classes.
 Therefore
$$|\mathcal{S}_{p}^G(N)|=\sum_{i=1}^{\infty} p^i w_1(p^{i-1}) + \sum_{i=0}^{\infty}p^i p w_2(p^i)=$$

$$=\sum_{i=1}^{\infty} p^i w_1(p^{i-1})+ \sum_{i=1}^{\infty}p^i w_2(p^{i-1})= p(\sum_{i=1}^{\infty}(p^{i-1}w_1(p^{i-1})+ p^{i-1}w_2(p^{i-1}))=$$
$$=p\sum_{i=1}^{\infty}p^{i-1} w_{\overline{N}}^{\overline{G}}(p^{i-1})=p|\mathcal{S}_{p}^{\overline{G}}(\overline{N})|, $$
so Eq. (\ref{eq}) is proved.

As $M\leq {\bf Z}(G)$, then the definition of the hypercentre implies that ${\bf Z}_\infty(\overline{G})=\overline{{\bf Z}_{\infty}(G)}$. Hence ${\bf Z}_{\infty}(\overline{G})\cap \overline{N}=({\bf Z}_{\infty}(G)\cap N)/M$. Then, we  utilize the inductive hypothesis and Eq. (\ref{eq}) to get
$$|{\bf Z}_{\infty}(G)\cap N|_p=p |{\bf Z}_{\infty}(\overline{G})\cap \overline{N}|_p=p |\mathcal{S}_{p}^{\overline{G}}(\overline{N})|_p= |\mathcal{S}_{p}^G(N)|_p,$$
and the proof of (a) is finished.

For proving (b), notice first that $|\mathcal{S}_{p}^G(N)|_p=|N|_p$ implies by (a)  that $|{\bf Z}_\infty(G)\cap N|_p=|N|_p$. Then, it trivially follows that a Sylow $p$-subgroup of $N$ is hypercentral in $G$. Conversely, if a Sylow $p$-subgroup $P$ of $N$ lies in ${\bf Z}_\infty(G)$, then
$$|P|=|N|_p\leq |{\bf Z}_\infty(G)\cap N|_p,$$
 so $|N|_p=|{\bf Z}_\infty(G)\cap N|_p$. By applying (a) we obtain the equality $|\mathcal{S}_{p}^G(N)|_p=|N|_p$, so (b) is proved.

 Now,  $N$ is hypercentral in $G$ if and only if every Sylow subgroup of $N$ is hypercentral in $G$, and this can be established by (b) because $|\mathcal{S}_{p}^G(N)|_p$ and $|N|_p$ are given by the $G$-class size frequency function of $N$. Therefore, the first assertion of the theorem is proved.
  \end{proof}

\noindent
{\it Remark.} Observe that $|N\cap {\bf Z}(G)|= w^G_N(1)$ and likewise, by Theorem \ref{th1}(a),  $|N\cap {\bf Z}_{\infty}(G)|$ can  be retrieved  from the $G$-class frequency function of $N$. We stress that, however, it is not possible in general to compute the intermediate terms $|N\cap {\bf Z}_i(G)|$ when $i>1$. This follows straightforwardly from \cite[Example 1]{CHM}. Nonetheless, we will give an easier example at the end of this section.
On the other hand, we note that  Theorem \ref{th1}(a) shows that a Sylow $p$-subgroup of $N$ is hypercentral in $G$ if and only if  $|\mathcal{S}_{p}^G(N)|_p$ is as large as possible.

 \begin{proof}[Proof of Corollary B]

By Theorem \ref{th1}(b), we know that the $G$-class size frequency function of $N$ determines whether the Sylow $p$-subgroups of $N$ lie in ${\bf Z}_{\infty}(G)$. In fact, when this occurs then $N$ has only one Sylow  $p$-subgroup. Thus, by Lemma (\ref{l2}), this function detects whether all $p$-elements of $N$ centralize each $p'$-element of $G$. Since ${\bf O}^p(G)$ is generated by  all $p'$-elements of $G$, the result follows.
 \end{proof}

We state again Theorem C in a more convenient manner.
\begin{theorem}
Suppose that $G$ is a finite group and $N$ is a normal subgroup of $G$. Then the $G$-class sizes of $N$ and their multiplicities determine whether all elements of $N$ lying in the centre of some Sylow subgroup of $G$ belong to ${\bf Z}(G)$. More precisely, if $p$ is a prime number, then

\begin{itemize}
\item[(a)] $|N\cap {\bf Z}(G)|$ divides $|\mathcal{S}_{p'}^G(N)|$.
\item[(b)] $|{\bf C}_N({\bf O}^{p'}(G))|$  divides $|\mathcal{S}_{p'}^G(N)|$.
\item[(c)] $|\mathcal{S}_{p'}^G(N)|_p=|N\cap {\bf Z}(G)|_p$ if and only if ${\bf Z}(P)\cap N \leq {\bf Z}(G)$, where $P$ is any Sylow $p$-subgroup of $G$.
\end{itemize}
\end{theorem}
  \begin{proof} To prove (a), we notice that an element $x\in N$ has $G$-class size a $p'$-number if and only if ${\bf C}_G(x)$ contains a Sylow $p$-subgroup of $G$.
  Thus, $$\mathcal{S}_{p'}^G(N)= \bigcup_{P\in {\rm Syl}_p(G)}{\bf C}_N(P).$$
   Since $N\cap {\bf Z}(G)\subseteq {\bf C}_N(P)$ for every $P\in {\rm Syl}_p(G)$, we conclude that
   $\mathcal{S}_{p'}^G(N)$  is a union of cosets of $N\cap {\bf Z}(G)$, whence we deduce (a).

As ${\bf O}^{p'}(G)$ coincides with the subgroup of $G$ generated by all Sylow $p$-subgroups of $G$, we have
$${\bf C}_N({\bf O}^{p'}(G))= \bigcap_{P\in {\rm Syl}_p(G)}{\bf C}_N(P)\subseteq \mathcal{S}_{p'}^G(N).$$
It follows that $\mathcal{S}_{p'}^G(N)$ is a union of cosets of ${\bf C}_N({\bf O}^{p'}(G))$, and this proves (b).

   For proving (c), we fix a Sylow $p$-subgroup $P$ of $G$ and put $Z:=P\cap N\cap {\bf Z}(G)$, which is a Sylow $p$-subgroup of $N\cap {\bf Z}(G)$ and normal in $G$. Next, we claim that if $n\in \mathcal{S}_{p'}^G(N)$
   then ${\bf C}_{G/Z}(nZ)={\bf C}_G(n)/Z$. Set $U/Z:={\bf C}_{G/Z}(nZ)$. It is clear that ${\bf C}_G(n)\leq U$. Furthermore, if $u\in U$, then we can write $n^u=nf(u)$ for some $f(u)\in Z$.
   Since $Z\leq {\bf Z}(G)$, identically as in the proof of Theorem \ref{th1}, it follows that $f$ is a group   homomorphism from $U$ to $Z$, whose kernel is exactly ${\bf C}_G(n)$.
   As a consequence, $U/{\bf C}_G(n)$ is isomorphic to a subgroup of $Z$. However, $Z$ is a $p$-group whereas $|U/{\bf C}_G(n)|$ divides $|n^G|$, which is a $p'$-number. This forces $U={\bf C}_G(n)$, so the claim is proved.

   We observe now that $\mathcal{S}_{p'}^G(N)$ is also a union of cosets of $Z$ and consider the set $\mathcal{\bar{S}}$ of such cosets and the group $P/Z$ acting by conjugation   on   $\mathcal{\bar{S}}$.
    By using the claim it easily follows that
    a coset $nZ\in {\mathcal{\bar{S}}}$ is fixed by $P/Z$  if and only if $n$ is fixed by $P$, or equivalently, $nZ\in {\bf C}_N(P)/Z$. Now, by applying the orbit formula to this action we get
    $$|\mathcal{\bar{S}}|= |\mathcal{S}_{p'}^G(N)|/|Z|\equiv |{\bf C}_N(P)|/|Z| \quad ({\rm mod } \, p ).$$
        On the other hand, since $|Z|=|N\cap {\bf Z}(G)|_p$ we have that the equality  $|\mathcal{S}_{p'}^G(N)|_p=|Z|=|N\cap {\bf Z}(G)|_p$
      holds if and only if $| \mathcal{S}_{p'}^G(N)|/|Z|$ is a $p'$-number, and this is equivalent to  $|{\bf C}_N(P)|/|Z|$ being a $p'$-number.
      Now, as ${\bf Z}(P)$ is a Sylow $p$-subgroup of ${\bf C}_G(P)$, then ${\bf Z}(P)\cap N$ is a Sylow $p$-subgroup of ${\bf C}_N(P)$. Thus, we deduce that
       $|{\bf C}_N(P)|/|Z|$ is a $p'$-number if and only if
      $Z={\bf Z}(P)\cap N$. Now it is easily seen that this condition is equivalent to the fact that ${\bf Z}(P)\cap N \leq {\bf Z}(G)$, and therefore, the proof of (c) is finished.

      Now, since for every prime $p$, the numbers  $|\mathcal{S}_{p'}^G(N)|_p$ and $|N\cap {\bf Z}(G)|_p$ are provided by the $G$-class size frequency function of $N$, the first assertion of the theorem follows.
\end{proof}

\noindent
{\it Example.} An easy example illustrates that, as indicated in the Introduction, the fact that a normal subgroup is abelian cannot be read off from the set of its $G$-class sizes and their multiplicities. Indeed the same group may possess two normal subgroups showing this. Let $D_{16}=\langle x, y | x^8=y^2=1, x^y =x^{-1}\rangle$ and $D_8$ be the dihedral groups of order $16$ and $8$ respectively. Set $G= D_{16}\times D_8$, or equivalently $G=$ \textsf{SmallGroup}$(128,2011)$ in the library of small groups within  \textsf{GAP} \cite{gap}, and let us consider its normal subgroups  $\langle x\rangle$ and $D_8$. The set of $G$-class sizes of both subgroups, counting multiplicities,  is $\{1,1,2,2,2\}$, whilst $\langle x\rangle$ is abelian and $D_8$ is not. This example also serves to show, as pointed out in the Remark, that  from the $G$-class size frequency function of a normal subgroup $N$ we cannot compute the sizes of  all the terms $N\cap {\bf Z}_i(G)$ for every $i$. Indeed, ${\bf Z}_2(G)=\langle x^2\rangle\times D_8$, so $|\langle x \rangle \cap {\bf Z}_2(G)|= 4$ and $|D_{8}\cap {\bf Z}_2(G)|= 8$.

\bigskip
In view of the above example, we can affirm that the derived length of a solvable normal subgroup  as well as the nilpotency class of a nilpotent one are not determined from the $G$-class size frequency function.

\bigskip
\noindent
{\bf Acknowledgements}  This work is supported by  Proyecto CIAICO/2021/193,  Genera\-litat Valenciana, Spain.

\bigskip
\noindent
{\bf Declarations}

\bigskip
\noindent
{\bf Conflict of interest} The author declares that he has no financial interests.

\bibliographystyle{plain}

\end{document}